\documentclass{amsart}
\usepackage{amsmath,amsthm,amsfonts, amssymb,amscd}
\usepackage[all]{xy}

\newtheorem{theorem}{Theorem}[section]
\newtheorem{lemma}[theorem]{Lemma}
\newtheorem{proposition}[theorem]{Proposition}

\newcommand{\minusre}{\hspace{0.3em}\raisebox{0.3ex}{\sl \tiny /}\hspace{0.3em}}
\newcommand{\minusli}{\hspace{0.3em}\raisebox{0.3ex}{\sl \tiny $\setminus $}\hspace{0.3em}}
\newcommand{\lex}{\,\overrightarrow{\times}\,}
\newcommand{\lexphi}{\,\overrightarrow{\times_\phi}\,}
\newcommand{\lsemiphi}{\,\overrightarrow{\ltimes_\phi}\,}

\newcommand{\RDP}{\mbox{\rm RDP}}
\newcommand{\RIP}{\mbox{\rm RIP}}

\begin{document}
\title[Kite $n$-Perfect Pseudo Effect Algebras]{Kite $n$-Perfect Pseudo Effect Algebras}
\author[Michal Botur, Anatolij Dvure\v{c}enskij]{Michal Botur$^1$, Anatolij Dvure\v censkij$^{1,2}$}
\date{}%
\thanks{Both authors gratefully acknowledge  the support by GA\v{C}R 15-15286S.  AD thanks also Slovak Research and Development Agency under contract APVV-0178-11,  grant VEGA No. 2/0059/12 SAV}
\address{$^1$Palack\' y University Olomouc, Faculty of Sciences, t\v r. 17.listopadu 12, CZ-771 46 Olomouc, Czech Republic}
\address{$^2$Mathematical Institute, Slovak Academy of Sciences, \v{S}tef\'anikova 49, SK-814 73 Bratislava, Slovakia}
\email{michal.botur@upol.cz, dvurecen@mat.savba.sk}

\keywords{Effect algebra, pseudo effect algebra, po-group, kite pseudo effect algebra, kite $n$-perfect pseudo effect algebra, Riesz decomposition property, subdirect irreducibility }
\subjclass[2010]{03G12,  81P15,  03B50}

\maketitle

\begin{abstract}
Kite pseudo effect algebras were recently introduced as a class of interesting examples of pseudo effect algebras using a po-group, an index set and two bijections on the index set.
We represent kite pseudo effect algebras with a special kind of the Riesz decomposition property as an interval in a lexicographic extension of the po-group which solves an open problem on representation of kites. In addition, we introduce kite $n$-perfect pseudo effect algebras and we characterize subdirectly irreducible algebras which are building stones of the theory.
\end{abstract}

\section{Introduction}

Theory of quantum structures is a mathematical theory motivated by mathematical foundations of quantum mechanics. It goes back to a fundamental paper by Birkhoff and von Neumann, \cite{BiNe}. The so-called Hilbert space quantum mechanics is modeled by the system $\mathcal L(H)$ of all closed subspaces of a real, complex or quaternionic Hilbert space $H$. The events in $\mathcal L(H)$ have a yes-no character. 20 years ago there was presented a new model, effect algebras, introduced in \cite{FoBe}, whose the most physically and mathematically important example is the system $\mathcal E(H)$ of all Hermitian operators of a Hilbert space $H$ which are between the zero operator, $O$, and the identity, $I$. If $\mathcal B(H)$ is the system of Hermitian operators on $H$, then $\mathcal B(H)$ is a partially ordered group (= po-group), where for two Hermitian operators $A$ and $B$ we write $A \le B$ iff $(Ax,x)\le (Bx,x)$ for all unit vectors $x \in H$. Then the events of $\mathcal E(H)$ have a fuzzy character, i.e. their spectra are subsets of the real interval $[0,1]$. In addition, observables in $\mathcal E(H)$ correspond to POV-measures, whereas in $\mathcal L(H)$ they correspond to  projector-valued measures and correspondingly to self-adjoint operators.

In general, effect algebras are partial algebra where the basic notion is addition $+$ of two mutually excluding events, equivalently, $a+b$ means disjunction of two mutually excluding events $a$ and $b$. Another important mathematical feature of effect algebras is the fact that many of effect algebras are intervals in the positive cone of po-groups. This is possible e.g. whenever the effect algebra satisfies the Riesz Decomposition Property, i.e. a property when every two decompositions of the same element have a joint refinement. A special class of effect algebras are MV-algebras, and they play an analogous role for effect algebras as Boolean algebras do for orthomodular lattices and posets. For more information on effect algebras, we recommend \cite{DvPu}.

Nevertheless neither RDP does hold for  the effect algebra $\mathcal E(H)$ of a separable complex Hilbert space $H$ nor $\mathcal E(H)$ is a lattice, $\mathcal E(H)$ can be covered by a system of MV-algebras, see \cite{Pul}. In other words, RDP does not hold globally for $\mathcal E(H)$ rather locally, \cite[Thm 3.4]{DvuR}.

Recently, a non-commutative generalization of effect algebras, called {\it pseudo effect algebras}, was introduced in \cite{DvVe1, DvVe2}. They generalize a non-commutative generalization of MV-algebras called {\it pseudo MV-algebras} in \cite{GeIo} or, equivalently, {\it GMV-algebras} in \cite{Rac}. In these algebras there is not more assumed that $+$ is necessarily commutative. Some physical motivations for pseudo effect algebras with possible physical situations in quantum mechanics were presented in \cite{DvVe4}. Also in such a case, some pseudo effect algebras are intervals in po-groups not necessarily Abelian. This is true e.g. when a pseudo effect algebra satisfies a stronger type of RDP, namely RDP$_1$.

Non-commutative operations, for example multiplication of matrices, are well-known both in mathematics and physics and their applications. In particular, the class of square  matrices of the form
$$ A(a,b)=
\left( \begin{array}{cc}
a & b \\
0& 1
\end{array}
\right)
$$
for $a>0$, $b \in (-\infty,\infty)$ with usual multiplication of matrices is a non-commutative linearly ordered group with the neutral element $A(1,0)$ and with the positive cone consisting of matrices $A(a,b)$ with $a>1$ or $a=1$ and $b\ge0$. It gives an example of a pseudo effect algebra with a strong type of RDP, namely RDP$_2$; in addition it is an example of a pseudo MV-algebra. We note that $A(a,b)$ is an extension of real numbers: If $b=0$, then $A(a,0)$ is a positive real number and if $b\ne 0$, then $A(a,b)$ denotes some kind of a generalized number (non-standard number) such that $A(a,b)$ is infinitely close to $A(a,0)$ but bigger than $A(a,0)$, and similarly if $b<0$, then $A(a,b)$ is also infinitely closed to $A(a,0)$ but smaller than $A(a,0)$, \cite{Haj}.

Every theory is so good as it possesses a large class of interesting examples. It is worthy to note that also starting with the positive and negative cone $G^+$ and $G^-$ of a po-group $G$, Abelian or non-Abelian, with an index set $I$ and with two bijections $\lambda,\rho: I\to I$ and with the ordinal sum $(G^+)^I$ at bottom and with $(G^-)^I$ above, we can obtain a class of pseudo effect algebras called in \cite{DvuK} {\it kite pseudo effect algebras}. Their names was motivated by an example from \cite{JiMo} which resembles the shape of a kite.  Their basic properties were developed in \cite{DvuK, DvHo}. We note that starting with an Abelian po-group, e.g. the group of integers, the resulting algebraic structure is not necessarily commutative.

In \cite{DvuK}, there was shown that if $G$ is a po-group with RDP$_1$, then the corresponding kite pseudo effect algebra $K_{I}^{\lambda,\rho}(G)$ has also RDP$_1$, so it has to be an interval in some unital po-group $(H,u)$ with RDP$_1$. But it was unknown how such a unital po-group looks like. Motivating by this problem, we present in this paper its solution, and we show that such a po-group is a special form of the so-called lexicographic extension $G$, see \cite[p. 11]{Gla}.

In addition, we introduce a new class of pseudo effect algebras, called kite $n$-perfect pseudo effect algebras. Such an algebra is characterized by a property that it can be decomposed into $(n+1)$ comparable slices. Special algebras with this property were studied e.g. in \cite{DXY, DvKr, DvXi} under the name $n$-perfect pseudo effect algebras; in this special case $|I|=1$. For kite $n$-perfect pseudo effect algebras, we study the building blocks of the theory~-- subdirectly irreducible kite $n$-perfect pseudo effect algebras.

The paper is organized as follows. The second section gathers the basic notions and results on the theory of pseudo effect algebras and po-groups. The third section introduces kite pseudo effect algebras. It presents a solution to a problem on a characterization of a kite by an interval in a po-group with RDP$_1$. We show that the solution is connected with a so-called lexicographic extension of the po-group. Therefore, the main effort is concentrated to show that such a lexicographic extension satisfies RDP$_1$. This important result is proved in two different ways. Finally, Section 4 introduces kite $n$-perfect pseudo effect algebras. These algebras generalize kite pseudo effect algebras; more precisely, a kite $n$-perfect pseudo effect algebra is a kite pseudo effect algebra iff $n=1$. For this class we will study subdirectly irreducible algebras which are building elements of the theory of quantum structures. Finally, Section 5 gives concluding remarks.

\section{Basic Notions on Pseudo Effect Algebras}

A non-commutative extension of effect algebras are pseudo effect algebras. We note that according to \cite{DvVe1, DvVe2}, a {\it pseudo effect algebra} is  a partial algebra  $ E=(E; +, 0, 1)$, where $+$ is a partial binary operation and $0$ and $1$ are constants, such that for all $a, b, c
\in E$, the following holds

\begin{enumerate}
\item[(i)] $a+b$ and $(a+b)+c$ exist if and only if $b+c$ and
$a+(b+c)$ exist, and in this case $(a+b)+c = a+(b+c)$;

\item[(ii)]
  there is exactly one $d \in E$ and
exactly one $e \in E$ such that $a+d = e+a = 1$;

\item[(iii)]
 if $a+b$ exists, there are elements $d, e
\in E$ such that $a+b = d+a = b+e$;

\item[(iv)] if $1+a$ or $a+1$ exists, then $a = 0$.
\end{enumerate}

If we define $a \le b$ if and only if there exists an element $c\in
E$ such that $a+c =b,$ then $\le$ is a partial ordering on $E$ such
that $0 \le a \le 1$ for any $a \in E.$ It is possible to show that
$a \le b$ if and only if $b = a+c = d+a$ for some $c,d \in E$. We
write $c = a \minusre b$ and $d = b \minusli a.$ Then

$$ (b \minusli a) + a = a + (a \minusre b) = b,
$$
and we write $a^- = 1 \minusli a$ and $a^\sim = a\minusre 1$ for any
$a \in E.$ Then $a^-+a=1=a+a^\sim$ and $a^{-\sim}=a=a^{\sim-}$ for any $a\in E.$

For basic properties of pseudo effect algebras see \cite{DvVe1, DvVe2}. We note that a pseudo effect algebra is an {\it effect algebra}  iff $+$ is commutative. Hence, this definition corresponds to the definition of an effect algebra introduced in \cite{FoBe}.

Many examples of pseudo effect algebras are connected with intervals in the positive cones of po-groups. We recall that a {\it po-group} (= partially ordered group) is a quintuple $G=(G;+,-,0,\le)$, where
$(G;+,-,0)$ is a group endowed with a partial order $\le$ such that if $a\le b,$ $a,b
\in G,$ then $x+a+y \le x+b+y$ for all $x,y \in G.$  We denote by
$G^+:=\{g \in G: g \ge 0\}$ and $G^-:=\{g \in G: g \le 0\}$ the {\it positive cone} and the {\it negative cone} of $G.$ If, in addition, $G$
is a lattice under $\le$, we call it an $\ell$-{\it group} (= lattice
ordered group). An element $u \in G^+$ is said to be a {\it strong unit} (or an {\it order unit}) if, given $g \in G,$ there is an integer $n \ge 1$ such that $g \le nu.$ The pair $(G,u),$ where $u$ is a fixed strong unit of $G,$ is said to be a {\it unital po-group}. We recall that  the {\it lexicographic product} of two po-groups $G_1$ and $G_2$ is the group $G_1\times G_2,$ where the group operations are defined by coordinates, and the ordering $\le $ on $G_1 \times G_2$ is defined as follows: For $(g_1,h_1),(g_2,h_2) \in G_1 \times G_2,$  we have $(g_1,h_1)\le (g_2,h_2)$  whenever (i) $g_1 <g_2$ or (ii) $g_1=g_2$ and $h_1\le h_2.$

We denote by  $\mathbb Z$ and $\mathbb R$ the commutative $\ell$-group of integers and the group of real numbers, respectively, and let $\mathbb N=\mathbb Z^+$.

A po-group $G$ is said to be {\it directed} if, given $g_1,g_2 \in G,$ there is an element $g \in G$ such that $g \ge g_1,g_2.$ Equivalently, $G$ is directed iff every element $g\in G$ can be expressed as a difference of two positive elements of $G.$ For example, every $\ell$-group or every  po-group with strong unit is directed.
For more information on po-groups and $\ell$-groups we recommend the books \cite{Dar, Fuc, Gla}.

Let  $(G;+,-,0,\le)$ be a po-group written in an additive form and fix $u \in G^+.$ If we set $\Gamma(G,u):=[0,u]=\{g \in G: 0 \le g \le u\},$ then
$$
\Gamma(G,u)=(\Gamma(G,u); +,0,u)\eqno(2.1)
$$
is a pseudo effect algebra, where $+$ is the restriction of the group addition $+$ to $[0,u],$ i.e. $a+b$ is defined in $\Gamma(G,u)$ for $a,b \in \Gamma(G,u)$ iff $a+b \in \Gamma(G,u).$ Then $a^-=u-a$ and $a^\sim=-a+u$ for any $a \in \Gamma(G,u).$ A pseudo effect algebra which is isomorphic to some $\Gamma(G,u)$ for some po-group $G$ with $u>0$ is said to be an {\it interval pseudo effect algebra}.

A pseudo effect algebra $E$ is said to be {\it  symmetric} if $a^-=a^\sim$ for each $a\in E.$ For example, if $G$ is a non-commutative po-group, then
$\Gamma(\mathbb Z \lex G, (1,0))$ is a symmetric pseudo effect algebra that is not an effect algebra.

Let $A,B$ be two subsets of a pseudo effect algebra $E$. We write (i) $A+B:=\{a+b: a\in A, b \in B, \mbox{ such that } a+b \mbox{ is defined in } E\}$, and we say that $A+B$ is {\it defined} if $a+b$ is defined in $E$ for all $a \in A$ and all $b \in B$. (ii) $A\preceq B$ if $a\le b $ for all $a \in A$ and all $b \in B$.

An {\it ideal} of a pseudo effect algebra $E$ is any non-empty subset $I$ of $E$ such that (i) if $x,y \in I$ and $x+y$ is defined in $E,$ then $x+y \in I,$ and (ii) $x\le y \in I$ implies $x\in I.$ An ideal $I$ is (i) {\it normal} if $x+I=I+x$ for any $x \in E,$ where $x+I:=\{x+y: y \in I,\ x+y $ exists in $E\}$ and in the dual way  we define $I+x$; (ii) {\it maximal} if it is a proper subset of $E$ and if $I$ is a subset of a proper ideal $J$ of $E$, then $I=J$.

A mapping $h$ from one pseudo effect algebra $E$ into the second one $F$ is said to be a {\it homomorphism} if (i) $h(1)=1$, (ii) if $a+b$ is defined in $E$ so is defined $h(a)+h(b)$ in $F$ and $h(a+b)=h(a)+h(b)$. A bijective homomorphism is an {\it isomorphism} if $h$ and $h^{-1}$ are homomorphisms.

A {\it state} on a pseudo effect algebra $E$ is a mapping $s: E\to [0,1]$ such that (i) $s(1)=1$, and (ii) $s(a+b)=s(a)+s(b)$. In other words, a state on $E$ is any homomorphism from $E$ into the effect algebra $\Gamma(\mathbb R,1)$. We note that there are stateless pseudo effect algebras.

The Riesz decomposition property for effect algebras or for Abelian po-groups is very important. It denotes the property that every two decompositions of the same element have a joint refinement decomposition. For pseudo effect algebras and po-groups not necessarily Abelian, the following kinds of the Riesz Decomposition properties were introduced in \cite{DvVe1,DvVe2}.

We say that an additively written  po-group $(G;+,-,0,\le)$ satisfies

\begin{enumerate}
\item[(i)]
the {\it Riesz Interpolation Property} (RIP for short) if, for $a_1,a_2, b_1,b_2\in G,$  $a_1,a_2 \le b_1,b_2$  implies there exists an element $c\in G$ such that $a_1,a_2 \le c \le b_1,b_2;$

\item[(ii)]
\RDP$_0$  if, for $a,b,c \in G^+,$ $a \le b+c$, there exist $b_1,c_1 \in G^+,$ such that $b_1\le b,$ $c_1 \le c$ and $a = b_1 +c_1;$

\item[(iii)]
\RDP\  if, for all $a_1,a_2,b_1,b_2 \in G^+$ such that $a_1 + a_2 = b_1+b_2,$ there are four elements $c_{11},c_{12},c_{21},c_{22}\in G^+$ such that $a_1 = c_{11}+c_{12},$ $a_2= c_{21}+c_{22},$ $b_1= c_{11} + c_{21}$ and $b_2= c_{12}+c_{22};$

\item[(iv)]
\RDP$_1$  if, for all $a_1,a_2,b_1,b_2 \in G^+$ such that $a_1 + a_2 = b_1+b_2,$ there are four elements $c_{11},c_{12},c_{21},c_{22}\in G^+$ such that $a_1 = c_{11}+c_{12},$ $a_2= c_{21}+c_{22},$ $b_1= c_{11} + c_{21}$ and $b_2= c_{12}+c_{22}$, and $0\le x\le c_{12}$ and $0\le y \le c_{21}$ imply  $x+y=y+x;$

\item[(v)]
\RDP$_2$  if, for all $a_1,a_2,b_1,b_2 \in G^+$ such that $a_1 + a_2 = b_1+b_2,$ there are four elements $c_{11},c_{12},c_{21},c_{22}\in G^+$ such that $a_1 = c_{11}+c_{12},$ $a_2= c_{21}+c_{22},$ $b_1= c_{11} + c_{21}$ and $b_2= c_{12}+c_{22}$, and $c_{12}\wedge c_{21}=0.$

\end{enumerate}

If, for $a,b \in G^+,$ we have for all $0\le x \le a$ and $0\le y\le b,$ $x+y=y+x,$ we denote this property by $a\, \mbox{\rm \bf com}\, b.$

The RDP will be denoted by the following table:

$$
\begin{matrix}
a_1  &\vline & c_{11} & c_{12}\\
a_{2} &\vline & c_{21} & c_{22}\\
  \hline     &\vline      &b_{1} & b_{2}
\end{matrix}\ \ .
$$

For Abelian po-groups, RDP, RDP$_1,$ RDP$_0$ and RIP are equivalent.

By \cite[Prop 4.2]{DvVe1} for directed po-groups, we have
$$
\RDP_2 \quad \Rightarrow \RDP_1 \quad \Rightarrow \RDP \quad \Rightarrow \RDP_0 \quad \Leftrightarrow \quad  \RIP,
$$
but the converse implications do not hold, in general.  A directed po-group $G$ satisfies \RDP$_2$ iff $G$ is an $\ell$-group, \cite[Prop 4.2(ii)]{DvVe1}.

We say that a pseudo effect algebra $E$ satisfies the above types of the Riesz decomposition properties, if in the definition of RDP's, we change $G^+$ to $E.$

A principal representation result on pseudo effect algebras with RDP$_1$ is the following result \cite[Thm 7.2]{DvVe2}:

\begin{theorem}\label{th:2.1}
For every pseudo effect algebra with \RDP$_1,$ there is a unique (up to isomorphism of unital po-groups) unital po-group $(G,u)$ with \RDP$_1$\ such that $E \cong \Gamma(G,u).$

In addition, $\Gamma$ defines a categorical equivalence between the category of unital po-groups with \RDP$_1$ and the category of pseudo effect algebras with \RDP$_1$.
\end{theorem}

A very important subclass of pseudo effect algebras consists of pseudo MV-algebras, a non-commutative generalization of MV-algebras, which were introduced independently in \cite{GeIo} as pseudo MV-algebras and in \cite{Rac} as generalized MV-algebras. We remind  that according to \cite{GeIo}, a {\it pseudo MV-algebra} is an algebra $(M;
\oplus,^-,^\sim,0,1)$ of type $(2,1,1,$ $0,0)$ such that the
following axioms hold for all $x,y,z \in M$ with an additional
binary operation $\odot$ defined via $$ y \odot x =(x^- \oplus y^-)
^\sim $$
\begin{enumerate}
\item[{\rm (A1)}]  $x \oplus (y \oplus z) = (x \oplus y) \oplus z;$

\item[{\rm (A2)}] $x\oplus 0 = 0 \oplus x = x;$

\item[{\rm (A3)}] $x \oplus 1 = 1 \oplus x = 1;$

\item[{\rm (A4)}] $1^\sim = 0;$ $1^- = 0;$

\item[{\rm (A5)}] $(x^- \oplus y^-)^\sim = (x^\sim \oplus y^\sim)^-;$

\item[{\rm (A6)}] $x \oplus (x^\sim \odot y) = y \oplus (y^\sim
\odot x) = (x \odot y^-) \oplus y = (y \odot x^-) \oplus
x;$\footnote{$\odot$ has a higher priority than $\oplus$.}

\item[{\rm (A7)}] $x \odot (x^- \oplus y) = (x \oplus y^\sim)
\odot y;$

\item[{\rm (A8)}] $(x^-)^\sim= x.$
\end{enumerate}
If we define $a\le b$ iff there is an element $c\in M$ such that $a\oplus c=b$, then $\le$ is a partial order and $a\vee b = (a^*\oplus b)^*\oplus b.$  With respect to this order, $M$ is a distributive lattice. We note that a pseudo effect algebra is an {\it MV-algebra} iff $\oplus$ is commutative.

For example, if $u$ is a strong unit of a (not necessarily Abelian)
$\ell$-group $G$,
$$\Gamma(G,u) := [0,u]
$$
and
\begin{eqnarray*}
x \oplus y &:=&
(x+y) \wedge u,\\
x^- &:=& u - x,\\
x^\sim &:=& -x +u,\\
x\odot y&:= &(x-u+y)\vee 0,
\end{eqnarray*}
then $\Gamma(G,u)=(\Gamma(G,u);\oplus, ^-,^\sim,0,u)$ is a pseudo MV-algebra.

The basic result on theory of pseudo MV-algebras \cite{Dvu1} is the following representation theorem.

\begin{theorem}\label{th:2.2}
Every pseudo MV-algebra is an interval $\Gamma(G,u)$ in a unique (up to isomorphism)  unital $\ell$-group $(G,u).$

In addition, the functor $\Gamma: (G,u) \mapsto \Gamma(G,u)$ defines a categorical equivalence between the category of unital $\ell$-groups and the variety of pseudo MV-algebras.
\end{theorem}

Now let $M$ be a pseudo MV-algebra. We define a partial operation $+$ on $M$ as follows, $a+b$ is defined iff $a\odot b = 0$; then $a+b:=a\oplus b$. Then $(M;+,0,1)$ is a lattice ordered pseudo effect algebra with RDP$_2$. Conversely, if $(E;+,0,1)$ is a lattice ordered pseudo effect algebra with RDP$_2$,   \cite[Thm 8.8]{DvVe2}, then $(E; \oplus,^-,^\sim,0,1),$ where
$$
a\oplus b := (b^-\minusli (a\wedge b^-))^\sim, \quad a,b \in E, \eqno(2.2)
$$
is a pseudo MV-algebra.

\section{Kite Pseudo Effect Algebras}

In this section, we present kite pseudo effect algebras. The main goal is to show how we can represent a kite pseudo effect algebra with RDP$_1$ as an interval in some unital po-group with RDP$_1$. We show that it can  be represented as an interval in some lexicographic extension of the group. To solve this problem, the main accent will be posed to establish the fact that this lexicographic extension satisfies RDP$_1$. We present even two different proofs of this fact.

For our aims,  we will use po-groups written in a multiplicative way.  Thus let $G=(G;\cdot,^{-1},e,\le)$ be a multiplicatively-written  po-group with an inverse $^{-1},$ identity element $e,$ and equipped with a partial order $\le$. Then $G^+:=\{g \in G\colon g\ge e\}$ and $G^-:=\{g\in G \colon g \le e\}$ are the corresponding positive and negative cones of $G$.

Let $I$ be a set. Define an algebra whose
universe is the set $(G^+)^I \uplus (G^-)^I,$ where $\uplus$ denotes a union of disjoint sets (concerning the element $e$ which is in both $G^+$ and $G^-$, see the comment below on convention).
We order its universe by keeping the original co-ordinatewise ordering
within $(G^+)^I$ and $(G^-)^I$, and setting $x\leq y$ for all
$x\in(G^+)^I$, $y\in(G^-)^I$. Then $\le$ is a partial order on $(G^+)^I \uplus (G^-)^I.$ Hence, the element $e^I:=\langle e: i \in I\rangle$ appears twice: at the bottom of $(G^+)^I$ and at the top of $(G^-)^I$. To
avoid confusion in the definitions below, we adopt a convention of writing
$a_i^{-1},b_i^{-1}, \dots$ for co-ordinates of elements of $(G^-)^I$ and
$f_i,g_i,\dots$ for co-ordinates of elements of $(G^+)^I$. In particular, we
will write $e^{-1}$ for $e$ as an element of $G^-,$  $e$ as an element of $G^+,$ and without loss of generality, we will assume that formally $e^{-1}\ne e.$ We also put $1$ for the
constant sequence $(e^{-1})^I:=\langle e^{-1}\colon i \in
I\rangle$ and $0$ for the constant sequence $e^I:=\langle e_i\colon i \in I\rangle$. Then $0$ and $1$ are the least and greatest elements of $(G^+)^I \uplus (G^-)^I.$

The following construction of kite pseudo effect algebras was presented in \cite[Thm 3.4]{DvuK}:

\begin{theorem}\label{th:3.1}
Let $G$ be a po-group and $\lambda,\rho:I\to I$ be bijections. Let us endow the set $K^{\lambda,\rho}_I(G):=(G^+)^I \uplus (G^-)^I$ with $0=e^I,$ $1=(e^{-1})^I$ and with a partial operation $+$ as follows:

$$\langle a_i^{-1}\colon i\in I\rangle + \langle b_i^{-1}\colon i\in I\rangle \eqno(I)$$
is not defined;

$$ \langle a_i^{-1}\colon i\in I\rangle + \langle f_i\colon i\in I\rangle:= \langle a_i^{-1}f_{\rho^{-1}(i)}\colon i\in I\rangle \eqno(II)
$$
whenever  $f_{\rho^{-1}(i)}\le a_i$ for all $i \in I;$

$$ \langle f_i\colon i\in I\rangle+ \langle a_i^{-1}\colon i\in I\rangle  := \langle f_{\lambda^{-1}(i)} a_i^{-1}\colon i\in I\rangle \eqno(III)
$$
whenever  $f_{\lambda^{-1}(i)}\le a_i$ for all $i \in I,$

$$
\langle f_i\colon i\in I\rangle + \langle g_i\colon i\in I\rangle:= \langle f_i g_i\colon i\in I\rangle
\eqno(IV)
$$
for all $\langle f_i\colon i\in I\rangle$ and all $\langle g_i\colon i\in I\rangle.$
Then the partial algebra $(K^{\lambda,\rho}_I(G); +,0,1)$ is a pseudo effect algebra.

For the negations in the kite $K_{I}^{\lambda,\rho}(G),$ we have
\begin{align*}
\langle a_i^{-1}\colon i\in I\rangle^\sim &= \langle a_{\rho(i)}\colon i\in I\rangle\\
\langle a_i^{-1}\colon i\in I\rangle^-&= \langle a_{\lambda(i)}\colon i\in I\rangle\\
\langle f_i\colon i\in I\rangle^\sim &= \langle f^{-1}_{\lambda^{-1}(i)}\colon i\in I\rangle\\
\langle f_i\colon i\in I\rangle^-&= \langle f^{-1}_{\rho^{-1}(i)}\colon i\in I\rangle.
\end{align*}

If $G$ is an $\ell$-group, then $K^{\lambda,\rho}_I(G)$ is a pseudo effect algebra with \RDP$_2,$ that is, a pseudo MV-algebra.
\end{theorem}

We note that if $\lambda$ and $\rho$ are identities on $I$, then the corresponding kite pseudo effect algebra $K^{\lambda,\rho}_I(G)$ is isomorphic to $\Gamma(\mathbb Z \lex G^I,(1,e^I))$. According to \cite[Thm 4.1]{DvuK}, the kite pseudo effect algebra satisfies RDP, or RDP$_1$, or RDP$_2$, respectively, iff $G$ satisfies RDP, or RDP$_1$, or RDP$_2$, respectively.

We note that if $O$ is a one-element po-group, then for any index set $I$, and bijections $\lambda,\rho$, the kite pseudo effect algebra $K^{\lambda,\rho}_I(O)$ is a two-element Boolean algebra.

The examples of kite pseudo effect algebras are presented in \cite{DvuK}. Now we present a new example which will be important for our aims. According to \cite[Ex 1.3.25]{Gla}, we present an example which is connected with a lexicographic extension of $G$.

Thus let $G$ be a po-group and $I$ be an index set. Take $\mathbb Z$ the group of integers and let $\phi:I\to I$ be a bijection. Let $H=\mathbb Z \lsemiphi G$ be the semidirect product of $G$ by $\mathbb Z$ defined as the set of elements $(n,\langle x_i: i \in I\rangle)\in \mathbb Z \times G^I$, where $n \in \mathbb Z$ and $\langle x_i:
i\in I\rangle \in G^I$ with the lexicographic ordering, that is, $(n,\langle x_i: i \in I\rangle)\le (m,\langle y_i: i \in I\rangle)$ iff $n<m$ or $n=m$ and $x_i\le y_i$ for every $i \in I$, and with the multiplication $*_\phi$ defined as follows
$$
(n,\langle x_i: i \in I\rangle)*_\phi (m,\langle y_i: i \in I\rangle)=
(n+m, \langle x_iy_{\phi^n(i)}: i\in I\rangle)\eqno(3.1)
$$
for all $(n,\langle x_i: i \in I\rangle), (m,\langle y_i: i \in I\rangle)\in \mathbb Z \lsemiphi G$. It is easy to verify, see \cite[Ex 1.3.25]{Gla}, that $\mathbb Z \lsemiphi G$ is a po-group with the neutral element $(0,e^I)$ and the inversion of the element $(n,\langle x_i: i \in I\rangle)$ is the element
$$ ((n,\langle x_i: i \in I\rangle))^{-1}=(-n, \langle x^{-1}_{\phi^{-n}(i)}: i \in I\rangle).
$$
Then the element $u=(0,e^I)$ is a strong unit of the po-group $\mathbb Z\lexphi G$. Hence,
$$
E_I^\phi(G):=\Gamma(\mathbb Z \lsemiphi G, u) \eqno(3.2)
$$
is a pseudo effect algebra. The kite pseudo effect algebra $K^{\lambda,\rho}_I(G)$, where $\lambda = Id_I$ is the identity on $I$, is isomorphic to the pseudo effect algebra $\Gamma(\mathbb Z \lsemiphi G,u)$, where $\phi = \rho^{-1}$.

In what follows, we show that every kite pseudo effect algebra $K^{\lambda,\rho}_I(G)$ is isomorphic to $E_\phi^I(G)$ for some bijection $\phi:I\to I$.

\begin{theorem}\label{th:3.2}
Every kite pseudo effect algebra  $K^{\lambda,\rho}_I(G)$   is isomorphic to the pseudo effect algebra $E_\phi^I(G)$ for some bijection $\phi: I \to I$.
\end{theorem}

\begin{proof}
Define a mapping $\Phi: K^{\lambda,\rho}_I(G) \to E_\phi^I(G)$, where $\phi = \rho^{-1}\circ \lambda$, as follows:
$\Phi(\langle a^{-1}_i: i \in I\rangle)= (1,\langle a^{-1}_i: i\in I\rangle)$ and $\Phi(\langle f_i:i \in I\rangle)= (0, \langle f_{\lambda^{-1}(i)}: i \in I\rangle)$. Then it is straightforward to verify that $\Phi$ is an isomorphism of pseudo effect algebras  in question.
\end{proof}

Using Theorem \ref{th:3.2}, we obtain below an answer to an open question posed in \cite{DvuK}: Describe a unital po-group $(H,u)$ with RDP$_1$ such that $K_{I}^{\lambda,\rho}(G)\cong \Gamma(H,u)$ when $G$ satisfies RDP$_1$.

Before that, we prove the following results.

\begin{proposition}\label{pr:3.3}
Let $G$ be a directed po-group.
The unital po-group $\Gamma(\mathbb Z \lsemiphi G,u)$ satisfies $\RIP$\, if and only if $G$ satisfies $\RIP$.
\end{proposition}

\begin{proof}
Let $\Gamma(\mathbb Z \lsemiphi G,u)$ satisfy RIP, and let $a_1,a_2 \le b_1,b_2$.  Fix $i_0 \in I$ and let $a^j_i = e=b^j_i$ if $i\ne i_0$ otherwise $a^j_{i_0}=a_j$ and $b^{j}_{i_0}= b_j$ for $j=1,2$. Then $(0,\langle a^j_i: i \in I\rangle)\le (0,\langle b^j_i: i \in I\rangle)$, so that, for every $i \in I$, there is $c_i\in G$ such that $(0,\langle a^j_i: i \in I\rangle)\le (0,\langle c_i: i \in I\rangle)\le (0,\langle b^j_i: i \in I\rangle)$. In particular, $a_1,a_2 \le c_{i_0}\le b_1,b_2$ as stated.

Conversely, let $G$ satisfy RIP and let $(n_j,\langle a^j_i: i \in I\rangle)\le (m_j,\langle b^j_i: i \in I\rangle)$ for $j=1,2$. We can assume that $n_1\le n_2\le m_1\le m_2$.

(i) Let $n_2 < m_1$. Directness of $G$ implies that, for every $i \in I$, there is $c_i \in G$ such that $a^1_i,a^2_i \le c_i$. Then $(n_j,\langle a^j_i: i \in I\rangle)\le (n_2,\langle c_i: i \in I\rangle)\le (m_j,\langle b^j_i: i \in I\rangle)$ for $j=1,2$.

(ii) Let $n_2=m_1$. Then $a^2_i\le b_i^1$ for each $i \in I$. If $n_1 <n_2$, then $(n_j,\langle a^j_i: i \in I\rangle)\le (n_2,\langle a^2_i: i \in I\rangle)\le (m_j,\langle b^j_i: i \in I\rangle)$ for $j=1,2$. Now assume $n_1=n_2$. Then $a^1_i\le b^2_i$ for every $i \in I$. If $m_1 < m_2$, then $(n_j,\langle a^j_i: i \in I\rangle)\le (n_2,\langle c_i: i \in I\rangle)\le (m_1,\langle a^1_i: i \in I\rangle)\le (m_j,\langle b^j_i: i \in I\rangle)$ for $j=1,2$. Finally, we assume $n_1=n_2=m_1=m_2$. Then $a^1_i,a^2_i\le b^1_i,b^2_i$ for each $i \in I$, so that there is $c_i\in G$ such that $a^1_i,a^2_i\le c_i\le  b^1_i,b^2_i$. Hence, $(n_j,\langle a^j_i: i \in I\rangle)\le (n_2,\langle c_i: i \in I\rangle)\le  (m_j,\langle b^j_i: i \in I\rangle)$ for $j=1,2$.
\end{proof}

We note that if, for the equation $a_1+a_2=b_1+b_2$ in a pseudo effect algebra $E$, the following table

$$
\begin{matrix}
a_1  &\vline & c_{11} & c_{12}\\
a_{2} &\vline & c_{21} & c_{22}\\
  \hline     &\vline      &b_{1} & b_{2}
\end{matrix}\ \
$$
denotes some RDP type decomposition, then, for the equivalent identity $b_1+b_2=a_1+a_2$, the next table
$$
\begin{matrix}
b_1  &\vline & c_{11} & c_{21}\\
b_{2} &\vline & c_{12} & c_{22}\\
  \hline     &\vline      &a_{1} & a_{2}
\end{matrix}\ \
$$
gives the corresponding RDP type decomposition.

\begin{theorem}\label{th:3.4}
Let $G$ be a directed po-group and $\phi:I \to I$ be a bijection. Then the unital po-group $(\mathbb Z \lsemiphi G, u)$ satisfies \RDP\, (\RDP$_1$, \RDP$_2$, respectively) if and only if $G$ satisfies \RDP\, (\RDP$_1$, \RDP$_2$, respectively).
\end{theorem}

\begin{proof}
The pseudo effect algebra $\Gamma(\mathbb Z \lsemiphi G, u)$ consists of two parts: $\{(0,\langle x_i: i \in I\rangle): x_i \ge e$ for each $i\in I\}\cup \{(0,\langle y_i: i \in I\rangle): y_i \le e$ for each $i\in I\}$.

Assume that $(\mathbb Z \lsemiphi G, u)$ satisfies \RDP\, (\RDP$_1$, \RDP$_2$, respectively). Let $a_1a_1=b_1b_2$ holds for some $a_1,a_2,b_1,b_2 \in G^+$. Fix an element $i_0\in I$ and for $j=1,2,$ let us define $A_j=(0,\langle f_i^j \colon i\in I\rangle)$ by $f_i^j=a_j$ if $i=i_0$ and $f_i^j=e$ otherwise, $B_j= (0,\langle g_i^j \colon i\in I\rangle)$ by $g_i^j=b_j$ if $i=i_0$ and $g_i^j=e$ otherwise. Then $A_1 *_\phi A_2=B_1*_\phi B_2$ so that there are $E_{11}=(0,\langle e^{11}_i \colon i\in I\rangle),$ $E_{12}=(0,\langle e^{12}_i \colon i\in I\rangle),$ $E_{21}=(0,\langle e^{21}_i \colon i\in I\rangle),$ and $E_{22}=(0,\langle e^{22}_i \colon i\in I\rangle),$ such that $A_1=E_{11}*_\phi E_{12},$ $A_2=E_{21}*_\phi E_{22},$ $B_1=E_{11}*_\phi E_{21},$ and $B_2=E_{12}*_\phi E_{22}.$ Hence, we see that $G$ satisfies RDP. In the same manner we deal with the other kinds of RDPs.

Conversely, let $G$ satisfy RDP$_1$. (The case when $G$ satisfies RDP deals in a similar way). To be more compact, we will write
$(n,\langle x_i\rangle)$ instead of $(n,\langle x_i: i \in I\rangle)$.

First, we show that the pseudo effect algebra $\Gamma(\mathbb Z \lsemiphi G, u)$ satisfies RDP$_1$.

We have the following cases.

(i) $(0, \langle  x_i \rangle) *_\phi (0,\langle  y_i \rangle)=
(0,\langle  u_i \rangle) *_\phi (0,\langle  v_i \rangle)$ for $x_i,y_i,u_i,v_i \ge 0$ for each $i \in I$. Then $x_iy_i=u_iv_i$ so that for them the RDP$_1$-decomposition $e_i^{11},e_i^{12},e_i^{21}, e_i^{22}$ in $G$ can be found, so that $(0,\langle  e_i^{11} \rangle), (0,\langle  e_i^{12} \rangle), (0,\langle  e_i^{21} \rangle), (0,\langle  e_i^{22} \rangle)$ is the corresponding RDP$_1$-decomposition of (i) in $\Gamma(\mathbb Z \lsemiphi G,u)$.

(ii) $(0, \langle  x_i \rangle) *_\phi (1,\langle  y_i \rangle)=
(0,\langle  u_i \rangle) *_\phi (1,\langle  v_i \rangle)$ for $x_i,u_i \ge 0$, $y_i,v_i \le e$ for each $i \in I$. Then $x_iy_i=u_iv_i$ for each $i \in I$. Since $G$ is directed, for any $i\in I$, there is an element $d_i \in G$ such that $y_i,v_i \ge d_i$. Then $x_iy_id_i^{-1}=u_iv_id_i^{-1}$ and for them we have the RDP$_1$ decomposition

$$
\begin{matrix}
x_i  &\vline & c_i^{11} & c_i^{12}\\
y_id_i^{-1} &\vline & c_i^{21} & c_i^{22}\\
  \hline     &\vline      &u_i & v_id_i^{-1}
\end{matrix}\ \ ,
$$
where $e_i^{12}\, \mbox{\rm \bf com}\, e_i^{21}$. Then

$$
\begin{matrix}
x_i  &\vline & c_i^{11} & c_i^{12}\\
y_i &\vline & c_i^{21} & c_i^{22}d_i\\
  \hline     &\vline      &u_i & v_i
\end{matrix}\ \
$$
and
$$
\begin{matrix}
(0,\langle x_i\rangle)  &\vline & (0,\langle c_i^{11}\rangle) & (0,\langle c_i^{12}\rangle)\\
(1,\langle y_i\rangle) &\vline & (0,\langle c_i^{21}\rangle) & (1,\langle c_i^{22}d_i\rangle)\\
  \hline     &\vline      &(0,\langle u_i\rangle) & (1,\langle v_i \rangle)
\end{matrix}\ \
$$
is an RDP$_1$ decomposition for (ii) in the pseudo effect algebra  $\Gamma(\mathbb Z \lsemiphi G,u)$.

(iii) $(1, \langle  x_i \rangle) *_\phi (0,\langle  y_i \rangle)=
(1,\langle  u_i \rangle) *_\phi (0,\langle  v_i \rangle)$ for $y_i,v_i \ge 0$, $x_i,u_i \le e$ for each $i \in I$. Then directness of $G$ implies that, for every $i \in I$, there is an element $d_i \in G$ such that $x_i,y_i,u_i,v_i \ge d_i$. Equality (iii) can be rewritten in the equivalent form $(1, \langle  d_i^{-1}x_i \rangle) *_\phi (0,\langle  y_id_i^{-1} \rangle)=
(1,\langle  d_i^{-1}u_i \rangle) *_\phi (0,\langle  v_id_i^{-1} \rangle)$ which yields
$d_i^{-1}x_i y_{\phi(i)}d_{\phi(i)}^{-1}=d_i^{-1}u_iv_{\phi(i)}d^{-1}_{\phi(i)}$. It has an RDP$_1$ decomposition in the pseudo effect algebra  $\Gamma(\mathbb Z \lsemiphi G,u)$

$$
\begin{matrix}
d_i^{-1}x_i  &\vline & c_i^{11} & c_i^{12}\\
y_{\phi (i)}d^{-1}_{\phi(i)} &\vline & c_i^{21} & c_i^{22}\\
  \hline     &\vline      &d_i^{-1}u_i & v_{\phi(i)}
d^{-1}_{\phi(i)}\end{matrix}\ \ ,
$$
consequently,
$$
\begin{matrix}
x_i  &\vline & d_ic_i^{11} & c_i^{12}\\
y_{\phi (i)} &\vline & c_i^{21} & c_i^{22}d_{\phi(i)}\\
  \hline     &\vline      &u_i & v_{\phi(i)}
\end{matrix}\ \ ,
$$
and it gives an RDP$_1$ decomposition of (iii) in the unital po-group $(\mathbb Z \lsemiphi G,u)$

$$
\begin{matrix}
(1,\langle x_i\rangle)  &\vline & (1,\langle d_ic_i^{11}\rangle) & (0,\langle c_{\phi^{-1}(i)}^{12}\rangle)\\
(0,\langle y_{i}\rangle) &\vline & (0,\langle c_{\phi^{-1}(i)}^{21}\rangle) & (0,\langle c_{\phi^{-1}(i)}^{22}d_i\rangle)\\
  \hline     &\vline      &(1,\langle u_i\rangle) & (0,\langle v_i \rangle)
\end{matrix}\ \ .
$$

(iv) $(1, \langle  x_i \rangle) *_\phi (0,\langle  y_i \rangle)=
(0,\langle  u_i \rangle) *_\phi (1,\langle  v_i \rangle)$ for $x_i,v_i\in G $, $y_i,u_i \ge e$ for each $i \in I$.

Then $x_iy_{\phi(i)} = u_iv_{i}$ for each $i \in I$, which implies $u_i^{-1}x_i= v_i y^{-1}_{\phi(i)}$.  If we use the decomposition

$$
\begin{matrix}
(1,\langle x_i\rangle)  &\vline & (0,\langle u_i\rangle) & (1,\langle u_{i}^{-1}x_i\rangle)\\
(0,\langle y_{i}\rangle) &\vline & (0,e^I) & (0,\langle y_i\rangle)\\
  \hline     &\vline      &(0,\langle u_i\rangle) & (1,\langle v_{i} \rangle)
\end{matrix}\ \ ,
$$
we see that it gives an RDP$_1$ decomposition for (iv); trivially  $(0,e^I)\, \mbox{\rm \bf com}\, (1,\langle u_{i}^{-1}x_i\rangle)$.

Summing up cases (i)--(iv), we see that the pseudo effect algebra $\Gamma(\mathbb Z \lsemiphi G,u)$ satisfies RDP$_1$.

Since $G$ satisfies RDP$_1$, it satisfies also RIP, see \cite[Prop 4.2]{DvVe1}. By Proposition \ref{pr:3.3}, the unital po-group $(\mathbb Z \lsemiphi G,u)$ satisfies RIP, too. Applying \cite[Thm 3.6]{DvKr}, we have that the po-group $(\mathbb Z \lsemiphi G,u)$ satisfies RDP$_1$ as it was claimed.

Finally, let us assume that $G$ satisfies RDP$_2$. By \cite[Prop 4.2(ii)]{DvVe1}, a directed po-group $G$ satisfies \RDP$_2$ iff $G$ is an $\ell$-group. It is easy to verify that if $G$ is an $\ell$-group, so is $\mathbb Z \lsemiphi G$, hence, $\mathbb Z \lsemiphi G$ satisfies RDP$_2$
\end{proof}

In the following we present a direct proof of the Theorem \ref{th:3.4}.

\begin{theorem}\label{th:3.5}
Let $G$ be a directed po-group and $\phi:I \to I$ be a bijection. Then the unital po-group $(\mathbb Z \lsemiphi G, u)$ satisfies \RDP\, (\RDP$_1$, \RDP$_2$, respectively) if and only if $G$ satisfies \RDP\, (\RDP$_1$, \RDP$_2$, respectively).
\end{theorem}

\begin{proof}
The positive cone of the unital po-group $(\mathbb Z \lsemiphi G, u)$ is the set $\{(0,\langle x_i: i \in I\rangle): x_i \ge e$ for each $i\in I\}\cup \bigcup_{n=1}^\infty \{(n, \langle x_i: i \in I\rangle), \ x_i \in G \mbox{ for each } i \in I\}$.

One implication is according to Theorem \ref{th:3.4} evident.

Therefore, we assume that $G$ satisfies RDP$_1$. (The case when $G$ satisfies RDP deals in a similar way). To be more compact, we will write
$(n,\langle x_i\rangle)$ instead of $(n,\langle x_i: i \in I\rangle)$.  We have the following cases.

(i) $(0, \langle  x_i \rangle) *_\phi (0,\langle  y_i \rangle)=
(0,\langle  u_i \rangle) *_\phi (0,\langle  v_i \rangle)$ for $x_i,y_i,u_i,v_i \ge 0$ for each $i \in I$. The proof of this case is identical to the proof of case (i) in Theorem \ref{th:3.4}.

(ii) $(0, \langle  x_i \rangle) *_\phi (n,\langle  y_i \rangle)=
(0,\langle  u_i \rangle) *_\phi (n,\langle  v_i \rangle)$ for $x_i,u_i \ge 0$, $y_i,v_i \in G$ for each $i \in I$, and $n\ge 1$. Then $x_iy_i=u_iv_i$ for each $i \in I$. Since $G$ is directed, for any $i\in I$, there is an element $d_i \in G$ such that $y_i,v_i \ge d_i$. Then $x_iy_id_i^{-1}=u_iv_id_i^{-1}$ and for them we have the RDP$_1$ decomposition

$$
\begin{matrix}
x_i  &\vline & c_i^{11} & c_i^{12}\\
y_id_i^{-1} &\vline & c_i^{21} & c_i^{22}\\
  \hline     &\vline      &u_i & v_id_i^{-1}
\end{matrix}\ \ ,
$$
where $c_i^{12}\, \mbox{\rm \bf com}\, c_i^{21}$. Then

$$
\begin{matrix}
x_i  &\vline & c_i^{11} & c_i^{12}\\
y_i &\vline & c_i^{21} & c_i^{22}d_i\\
  \hline     &\vline      &u_i & v_i
\end{matrix}\ \
$$
and
$$
\begin{matrix}
(0,\langle x_i\rangle)  &\vline & (0,\langle c_i^{11}\rangle) & (0,\langle c_i^{12}\rangle)\\
(n,\langle y_i\rangle) &\vline & (0,\langle c_i^{21}\rangle) & (n,\langle c_i^{22}d_i\rangle)\\
  \hline     &\vline      &(0,\langle u_i\rangle) & (n,\langle v_i \rangle)
\end{matrix}\ \
$$
is an RDP$_1$ decomposition for (ii) in the po-group $\mathbb Z \lsemiphi G$.

(iii) $(n, \langle  x_i \rangle) *_\phi (0,\langle  y_i \rangle)=
(n,\langle  u_i \rangle) *_\phi (0,\langle  v_i \rangle)$ for $y_i,v_i \ge 0$, $x_i,u_i \in G$ for each $i \in I$, and $n \ge 1$. The directness of $G$ implies, for each $i \in I$, there is $d_i \in G$ such that $d_i\le x_i,y_i,u_i,v_i$. Equality (iii) can be rewritten in the equivalent form $(n, \langle  d_i^{-1}x_i \rangle) *_\phi (0,\langle  y_id_i^{-1} \rangle)=
(n,\langle  d_i^{-1}u_i \rangle) *_\phi (0,\langle  v_id_i^{-1} \rangle)$ which yields $d_i^{-1}x_iy_{\phi^n(i)}d^{-1}_{\phi^n(i)}=d_i^{-1}u_iv_{\phi^n(i)}d_{\phi^n(i)}^{-1}$.
It entails an RDP$_1$ decomposition in the po-group $G^I$

$$
\begin{matrix}
d_i^{-1}x_i  &\vline & c_i^{11} & c_i^{12}\\
y_{\phi^n (i)}d^{-1}_{\phi^n(i)} &\vline & c_i^{21} & c_i^{22}\\
  \hline     &\vline      &d_i^{-1}u_i & v_{\phi^n(i)}
d^{-1}_{\phi^n(i)}\end{matrix}\ \ ,
$$
consequently,
$$
\begin{matrix}
x_i  &\vline & d_ic_i^{11} & c_i^{12}\\
y_{\phi^n (i)} &\vline & c_i^{21} & c_i^{22}d_{\phi^n(i)}\\
  \hline     &\vline      &u_i & v_{\phi^n(i)}
\end{matrix}\ \ ,
$$
and it gives an RDP$_1$ decomposition of (iii) in the unital po-group $(\mathbb Z \lsemiphi G,u)$

$$
\begin{matrix}
(n,\langle x_i\rangle)  &\vline & (n,\langle d_ic_i^{11}\rangle) & (0,\langle c_{\phi^{-n}(i)}^{12}\rangle)\\
(0,\langle y_{i}\rangle) &\vline & (0,\langle c_{\phi^{-n}(i)}^{21}\rangle) & (0,\langle c_{\phi^{-n}(i)}^{22}d_i\rangle)\\
  \hline     &\vline      &(n,\langle u_i\rangle) & (0,\langle v_{i} \rangle)
\end{matrix}\ \ .
$$

(iv) $(n, \langle  x_i \rangle) *_\phi (0,\langle  y_i \rangle)=
(0,\langle  u_i \rangle) *_\phi (n,\langle  v_i \rangle)$ for $x_i,v_i\in G $, $y_i,u_i \ge e$ for each $i \in I$, $n \ge 1$.

Then $x_iy_{\phi^n(i)} = u_iv_{i}$ for each $i \in I$, which implies $u_i^{-1}x_i= v_i y^{-1}_{\phi^n(i)}$.   If we use the decomposition

$$
\begin{matrix}
(n,\langle x_i\rangle)  &\vline & (0,\langle u_i\rangle) & (n,\langle u_{i}^{-1}x_i\rangle)\\
(0,\langle y_{i}\rangle) &\vline & (0,e^I) & (0,\langle y_i\rangle)\\
  \hline     &\vline      &(0,\langle u_i\rangle) & (n,\langle v_{i} \rangle)
\end{matrix}\ \ ,
$$
we see that it gives an RDP$_1$ decomposition for (iv); trivially  $(0,e^I)\, \mbox{\rm \bf com}\, (n,\langle u_{i}^{-1}x_i\rangle)$.

(v) $(n, \langle  x_i \rangle) *_\phi (0,\langle  y_i \rangle)=
(m_1,\langle  u_i \rangle) *_\phi (m_2,\langle  v_i \rangle)$ for $x_i,u_i, v_i\in G $, $y_i\ge e$ for each $i \in I$, where $m_1,m_2\ge 1$ and $m_1+m_2=n$. Then $x_iy_{\phi^n(i)}=u_iv_{\phi^{m_1}(i)}$ for each $i\in I$. Hence, the following table gives an RDP$_1$ decomposition for (v)

$$
\begin{matrix}
(n,\langle x_i\rangle)  &\vline & (m_1,\langle u_i\rangle) & (m_2,\langle u^{-1}_{\phi^{-m_1}(i)}x_{\phi^{-m_1}(i)}\rangle)\\
(0,\langle y_{i}\rangle) &\vline & (0,e^I) & (0,\langle y_i\rangle)\\
  \hline     &\vline      &(m_1,\langle u_i\rangle) & (m_2,\langle v_{i} \rangle)
\end{matrix}\ \ .
$$

(vi) $(0, \langle  x_i \rangle) *_\phi (n,\langle  y_i \rangle)=
(m_1,\langle  u_i \rangle) *_\phi (m_2,\langle  v_i \rangle)$ for $y_i,u_i, v_i\in G $, $x_i\ge e$ for each $i \in I$, where $m_1,m_2\ge 1$ and $m_1+m_2=n$. Then (vi) can be rewritten in the form $(m_1,\langle  u_i \rangle) *_\phi (m_2,\langle  v_i \rangle) = (0, \langle  x_i \rangle) *_\phi (n,\langle  y_i \rangle)$ and for it we have $u_iv_{\phi^{m_1}(i)}=x_iy_i$ for each $i \in I$ and  the following RDP$_1$ decomposition

$$
\begin{matrix}
(m_1,\langle u_i\rangle)  &\vline & (0,\langle x_i\rangle) & (m_1,\langle x_i^{-1}u_i\rangle)\\
(m_2,\langle v_{i} \rangle) &\vline & (0,e^I) & (m_2,\langle v_i\rangle)\\
  \hline     &\vline      &(0,\langle x_i\rangle) & (n,\langle y_{i}\rangle)
\end{matrix}\ \ .
$$

(vii) $(n_1, \langle  x_i \rangle) *_\phi (n_2,\langle  y_i \rangle)=
(m_1,\langle  u_i \rangle) *_\phi (m_2,\langle  v_i \rangle)$ for $x_i, y_i,u_i, v_i\in G $, for each $i \in I$, where $n_1, n_2, m_1,m_2 \ge 1$, $n_1+n_2=n=m_1+m_2$ and $m_1> n_1$. Then $x_iy_{\phi^{n_1}(i)}=u_iv_{\phi^{m_1}(i)}$ for each $i \in I$, and (vii)  has the following RDP$_1$ decomposition

$$
\begin{matrix}
(m_1,\langle u_i\rangle)  &\vline & (n_1,\langle x_i\rangle) & (m_1-n_1,\langle x_{\phi^{-n_1}(i)}^{-1}u_{\phi^{-n_1}(i)}\rangle)\\
(m_2,\langle v_{i} \rangle) &\vline & (0,e^I) & (m_2,\langle v_i\rangle)\\
  \hline     &\vline      &(n_1,\langle x_i\rangle) & (n_2,\langle y_{i}\rangle)
\end{matrix}\ \ \mbox{ if } m_1 >n_1
$$
is an RDP$_1$ decomposition.

(viii) $(n_1, \langle  x_i \rangle) *_\phi (n_2,\langle  y_i \rangle)=
(m_1,\langle  u_i \rangle) *_\phi (m_2,\langle  v_i \rangle)$ for $x_i, y_i,u_i, v_i\in G $, for each $i \in I$, where $n_1, n_2, m_1,m_2 \ge 1$, $n_1+n_2=n=m_1+m_2$ and $n_1> m_1$. Then $x_i y_{\phi^{n_1}(i)}=u_iv_{\phi^{m_1}(i)}$ for each $i \in I$. Hence, the following table

$$
\begin{matrix}
(n_1,\langle x_i\rangle)  &\vline & (m_1,\langle u_i\rangle) & (n_1-m_1,\langle u_{\phi^{-m_1}(i)}^{-1}x_{\phi^{-m_1}(i)}\rangle)\\
(n_2,\langle y_{i} \rangle) &\vline & (0,e^I) & (n_2,\langle y_i\rangle)\\
  \hline     &\vline      &(m_1,\langle u_i\rangle) & (m_2,\langle v_{i}\rangle)
\end{matrix}\ \ \mbox{ if } n_1 >m_1
$$
gives an RDP$_1$ decomposition for (viii).

(ix) $(n_1, \langle  x_i \rangle) *_\phi (n_2,\langle  y_i \rangle)=
(m_1,\langle  u_i \rangle) *_\phi (m_2,\langle  v_i \rangle)$ for $x_i, y_i,u_i, v_i\in G $, for each $i \in I$, where $n_1, n_2, m_1,m_2 \ge 1$, $n_1+n_2=n=m_1+m_2$ and $n_1= m_1$.  Then $x_iy_{\phi^{n_1}(i)}= u_iv_{\phi^{n_1}(i)}$. The directness of $G$ entails that, for
every $i \in I$, there is $d_i\in G$ such that $x_i,y_i,u_i,v_i\ge d_i$. Hence, $d_i^{-1}x_iy_{\phi^{n_1}(i)}d^{-1}_{\phi^{n_1}(i)}= d_i^{-1}u_iv_{\phi^{n_1}(i)}d^{-1}_{\phi^{n_1}(i)}$. The RDP$_1$ holding in $G$, we have the following RDP$_1$ table

$$
\begin{matrix}
d_i^{-1}x_i  &\vline & c_i^{11} & c_i^{12}\\
y_{\phi^{n_1}(i)}d^{-1}_{\phi^{n_1}(i)} &\vline & c_i^{21} & c_i^{22}\\
  \hline     &\vline      &d_i^{-1}u_i & v_{\phi^{n_1}(i)}d^{-1}_{\phi^{n_1}(i)}
\end{matrix}\ \ ,
$$
so that

$$
\begin{matrix}
x_i  &\vline & d_ic_i^{11} & c_i^{12}\\
y_{\phi^{n_1}(i)} &\vline & c_i^{21} & c_i^{22}d_{\phi^{n_1}(i)}\\
  \hline     &\vline      &u_i & v_{\phi^{n_1}(i)}
\end{matrix}\ \ .
$$

It gives the RDP$_1$ decomposition of (ix)

$$
\begin{matrix}
(n_1,\langle x_i \rangle)  &\vline & (n_1,\langle d_ic_i^{11}\rangle) & (0,\langle c_{\phi^{-n_1}(i)}^{12}\rangle )\\
(n_2,\langle  y_i\rangle) &\vline & (0,\langle c_{\phi^{-n_1}(i)}^{21}\rangle) & (n_2,\langle c_{\phi^{-n_1}(i)}^{22}d_i\rangle)\\
  \hline     &\vline      &(n_1,\langle u_i \rangle)& (n_2,\langle v_i \rangle)
\end{matrix}\ \ .
$$

Now assume that $G_2$ satisfy RDP$_2$.  By \cite[Prop 4.2(ii)]{DvVe1}, a directed po-group $G$ satisfies \RDP$_2$ iff $G$ is an $\ell$-group. It is easy to verify that if $G$ is an $\ell$-group, so is $\mathbb Z \lsemiphi G$.
Another proof of this implication follows all previous steps (i)--(ix)  for RDP$_2$ assumptions which prove that the unital po-group $(\mathbb Z \lsemiphi G, u)$ has RDP$_2$.
\end{proof}

Finally, we present an answer to an open question posed in \cite{DvuK}: Describe a unital po-group $(H,u)$ with RDP$_1$ such that $K_{I}^{\lambda,\rho}(G)\cong \Gamma(H,u)$ when $G$ satisfies RDP$_1$.

\begin{theorem}\label{th:3.6}
Let $G$ be a directed po-group satisfying \RDP$_1$. For the kite pseudo effect algebra $K_{I}^{\lambda,\rho}(G)$, the unital po-group $(\mathbb Z \lsemiphi G, u)$, where $\phi= \rho^{-1}\circ \lambda$, is a unique (up to isomorphism) unital po-group with \RDP$_1$\, such that $K_{I}^{\lambda,\rho}(G)\cong \Gamma(\mathbb Z \lsemiphi G, u)$,
\end{theorem}

\begin{proof}
Since the po-group $G$ satisfies RDP$_1$, \cite[Thm 4.1]{DvuK}, the kite pseudo effect algebra $K_{I}^{\lambda,\rho}(G)$ satisfies RDP$_1$, too. By Theorem \ref{th:3.2}, the kite pseudo effect algebra $K_{I}^{\lambda,\rho}(G)$ is isomorphic to the pseudo effect algebra $\Gamma(\mathbb Z \lsemiphi G, u)$. According to Theorem \ref{th:3.4}, both the pseudo effect algebra $\Gamma(\mathbb Z \lsemiphi G, u)$ and the unital po-group $(\mathbb Z \lsemiphi G, u)$ satisfy RDP$_1$. Hence, by the Representation Theorem of pseudo effect algebras with RDP$_1$, Theorem \ref{th:2.1}, the kite pseudo effect algebra $K_{I}^{\lambda,\rho}(G)$ has an isomorphic representation by the unital po-group $(\mathbb Z \lsemiphi G, u)$ with RDP$_1$.
\end{proof}

A similar result we have for kite pseudo effect algebras $K_{I}^{\lambda,\rho}(G)$ when $G$ is an $\ell$-group.

\begin{theorem}\label{th:3.7}
If $G$ is a po-group with $\RDP_2$, then the kite pseudo effect algebra $K_{I}^{\lambda,\rho}(G)$ has a representation by the unital po-group $(\mathbb Z \lsemiphi G, u)$ with $\RDP_2$,  where $\phi= \rho^{-1}\circ \lambda$.
\end{theorem}

\begin{proof}
The result follows from an application of Theorem \ref{th:3.4}.
\end{proof}

Theorem \ref{th:3.7} represents also the kite pseudo MV-algebra $K_{I}^{\lambda,\rho}(G)$ when $G$ is an $\ell$-group:

\begin{theorem}\label{th:3.8}
If $G$ is an $\ell$-group, then the kite pseudo effect algebra $K_{I}^{\lambda,\rho}(G)$ has a representation by the unital $\ell$-group $(\mathbb Z \lsemiphi G, u)$,  where $\phi= \rho^{-1}\circ \lambda$.
\end{theorem}

\begin{proof}
By Theorem \ref{th:3.4}, the kite pseudo effect algebra $K_{I}^{\lambda,\rho}(G)$ has RDP$_2$ and the unital po-group $(\mathbb Z \lsemiphi G, u)$ is in fact an $\ell$-group. Therefore, the kite pseudo effect algebra $K_{I}^{\lambda,\rho}(G)$ can be converted into a pseudo MV-algebra using (2.2). Applying the basic Representation Theorem of pseudo MV-algebras, Theorem \ref{th:2.2}, we see that the kite effect algebra $K_{I}^{\lambda,\rho}(G)$ is isomorphic to the interval in the unital $\ell$-group $(\mathbb Z \lsemiphi G, u)$.
\end{proof}

As we have just noted, if $G$ is an $\ell$-group, the kite pseudo effect algebra can be converted into a pseudo MV-algebra; we call it a {\it kite pseudo MV-algebra}. Such pseudo MV-algebras were studied firstly in \cite{DvKo}.

\section{Kite $n$-perfect Pseudo Effect Algebras}

In this section, we introduce so-called kite $n$-perfect pseudo effect algebras. They generalize kites because a kite $n$-perfect pseudo effect algebra is a kite if and only if $n=1$. We describe some basic properties and we concentrate to characterize subdirectly irreducible kite $n$-perfect pseudo effect algebras.

Now let $u_n =(n,e^I)$ for each integer $n \ge 1$. Then the pseudo effect algebra
$$
E_{I,n}^{\phi}(G)=\Gamma(\mathbb Z \lsemiphi G, u_n)\eqno(4.1)
$$
can be decomposed into $n+1$ slices, $E_0,E_1,\ldots, E_n$, of the form $E_0=\{(0,\langle x_i: i \in I\rangle): x_i\ge e \mbox{ for every } i \in I \}$, $E_k:=\{(k, \langle x_i: i\in I\rangle): \ x_i \in G\}$ for $k=1,\ldots,n-1$ and $E_n=\{(n,
\langle x_i: i\in I \rangle): x_i \le e \mbox{ for any } i \in I\}.$

Then $E_0,E_1,\ldots,E_n$ have the following properties:
\begin{itemize}
\item[(i)] $\bigcup_{k=0}^n E_k = E,$ $E_k\cap E_j = \emptyset$ for $k\ne j$, $k,j=0,\ldots, n$,
\item[(ii)] $E_k\preceq E_j$ whenever $k<j$,
\item[(iii)] $E_k+E_j$ is defined and $E_k+E_j = E_{k+j}$ whenever $k+j<n$,
\item[(iv)] $E_k+E_j$ is not defined in $E$ whenever $k+j>n$,
\item[(v)] $(k, \langle x_i: i\in I\rangle)^- = (n-k, \langle x^{-1}_{\phi^{n-k}(i)}: i \in I\rangle)$, $(k, \langle x_i: i\in I\rangle)^\sim = (n-k, \langle x^{-1}_{\phi^{-k}(i)}: i \in I\rangle)$,
and $E_k^-= E_{n-k}  =E_k^\sim $ for $k=0,1,
\ldots, n$,
\item[(vi)] $E_0$ is a unique ideal of $E$ that is normal and maximal.

\item[(vii)] $E$ has a unique state $s$, namely, $s(E_i)=i/n$, $i=0,1,\ldots,n$.
\end{itemize}

If in particular $|I|=1$, then there is a unique bijection $\phi: I\to I$ namely, the identity on $I$. In such a case, (4.1) is simply an interval in the lexicographic product $\mathbb Z \lex G$, and such pseudo effect algebras were studied in \cite{DXY, DvKr, DvXi} under that name $n$-perfect pseudo effect algebras. They can be decomposed into $n+1$ comparable slides that roughly speaking satisfy the properties described by (i)--(vi) from the latter paragraph. Therefore, pseudo effect algebras of the form (3.2) will be said to be {\it kite $n$-perfect pseudo effect algebras}.

In \cite{DXY, DvKr, DvXi}, there are presented  algebraic conditions posed to a pseudo effect algebra $E$ in order to be isomorphic to a pseudo effect algebra of the form $\Gamma(\mathbb Z \lex G,(n,0))$ for some directed po-group $G$ with RDP$_1$. Motivating by these $n$-perfect pseudo effect algebras, we suggest to find algebraic conditions which guarantee that a pseudo effect algebra is isomorphic to some kite $n$-perfect pseudo effect algebra.

If $E$ is a pseudo MV-algebra, then there is a one-to-one correspondence between congruences and normal ideals of pseudo MV-algebras. For pseudo effect algebras this correspondence is more delicated. However, for pseudo effect algebras with RDP$_1$ there is an analogous direct relationship.

We note that an equivalence $\sim$ on  a pseudo effect algebra $E$ is a {\it congruence}, if for $a_1,a_2,b_1,b_2$ such that $a_1\sim a_2$ and $b_1 \sim b_2$, the existence $a_1+b_1$ and $a_2+b_2$ in $E$ implies $a_1+b_1 \sim a_2+b_2.$  If $I$ is a normal ideal of $E$ with RDP$_1$, then the relation $\sim_I$ defined on $E$ by
$a\sim_I b$ iff there are $e,f \in I$ such that $a\minusli e=b\minusli f$ is a congruence on $E$, and $E/\sim_I$ is again a pseudo effect algebra with RDP$_1$, \cite[Prop 3.1]{DvVe3}, \cite[Prop 4.1]{185}.

We remind that by an {\it o-ideal} of a directed po-group $G$ we understand any normal directed convex subgroup $H$ of $G$. If $G$ is a po-group, so is $G/H,$ where $x/H \le y/H$ iff $x\le h_1+y$ for some $h_1\in H$ iff $x \le y+h_2$ for some $h_2 \in H.$
If $G$ satisfies one of RDP's, then $G/H$ satisfies the same RDP, \cite[Prop 6.1]{174}.

We note that a pseudo effect algebra $E$ is a {\it subdirect product} of a system of pseudo effect algebras $(E_t: t \in T)$ if there is an injective homomorphism $h: E\to \prod_{t\in T}E_t$ such that $\pi_t(h(E))=E_t$ for each $t \in T,$ where $\pi_t$ is the $t$-th projection of $\prod_{t \in T} E_t$ onto $E_t.$ In addition, $E$ is {\it subdirectly irreducible} if whenever $E$ is a subdirect product of $(E_t: t \in T),$ there exists $t_0 \in T$ such that $\pi_{t_0}\circ h$ is an isomorphism of pseudo effect algebras.

Hence, subdirect irreducibility of kite $n$-perfect pseudo effect algebras with RDP$_1$ will be studied in the relation between the least non-trivial (i.e. $\ne \{0\}$) normal ideals of kite pseudo effect algebras and the least non-trivial (i.e. $\ne \{e\}$) o-ideals of the corresponding po-groups.

In what follows, we show when a kite $n$-perfect pseudo effect algebra is subdirectly irreducible and its relation to subdirect irreducibility of the original po-group. These questions were studied in \cite{DvuK, DvHo} for kite pseudo effect algebras. We note that our proofs follow some ideas from \cite{DvuK, DvHo} improved for our situation.

Let $\alpha$ be a cardinal.  An element $(k,\langle x_i\colon i \in I\rangle) \in \Gamma(\mathbb Z \lsemiphi G, u_n)$ is said to be $\alpha$-{\it dimensional}  if $|\{i \in I: x_i \ne e\}|=\alpha.$ One-dimensional elements are particulary easy to work with. 

\begin{proposition}\label{pr:6.3}
Let $I$ be a set and $\phi:I \to I$ be a bijection.
If $H$ is an o-ideal of a directed po-group $G$ and $N=H^+,$ then $N^I:=\{(0,\langle x_i\colon i\in I\rangle): x_i \in N,\ i\in I\}$ is a normal ideal of the kite $n$-perfect pseudo effect algebra
$\Gamma(\mathbb Z \lsemiphi G, u_n)$.  In addition, if $N_f^I$ denotes the set of all finite dimensional elements from $N^I,$ then $N_f^I$ is a non-trivial normal ideal of the kite pseudo effect algebra.

Conversely, if $J$ is a normal ideal of $\Gamma(\mathbb Z \lsemiphi G, u_n)$, $J \subseteq N^I$, then $\pi_i(J)$ is the positive cone of an o-ideal of $H,$ where
$\pi_i$ is the $i$-th projection of $(0,\langle x_i \colon i \in I\rangle) \mapsto x_i.$
\end{proposition}

\begin{proof}
It is evident that $N^I$ is an ideal of the kite $n$-perfect pseudo effect algebra. To show that normality of $N^I$, we have to verify $x+N^I=N^I+x$ for any $(k,\langle x_i: i \in I\rangle)\in E_k,$ $k=0,1,\ldots,n$. This means
to show $(k,\langle x_i)+(0,\langle a_i\rangle)=(0,\langle b_i \rangle)+(k,\langle x_i\rangle)$ for $a_i, b_i\ge e$ $(i\in I)$. This means $x_ia_{\phi^k(i)}=b_ix_i$. Using normality of $N$, we see that $N^I$ is normal, too.

The same is true also for $N^I_f$ which means that $N^I_f$ is a non-trivial normal ideal of the kite $n$-perfect pseudo effect algebra.
\end{proof}

\begin{proposition}\label{pr:6.5}
Let $\Gamma(\mathbb Z \lsemiphi G, u_n)$, a kite $n$-perfect pseudo effect algebra of a po-group $G$, have the least non-trivial normal ideal. Then $G$ is subdirectly irreducible.
\end{proposition}

\begin{proof}
Suppose the converse, i.e. $G$ is not subdirectly irreducible. Then there exists a set $\{H_t: t \in T\}$ of non-trivial o-ideals of $G$ such that $\bigcap_{t \in T} H_t =\{e\}.$ According to Proposition \ref{pr:6.3}, let us define $N_t=H_t^+,$ $t \in T.$ Then every $N_t^I$ is a normal ideal of the kite $n$-perfect pseudo effect algebra $\Gamma(\mathbb Z \lsemiphi G, u_n)$. Hence, $\bigcap_{t \in T} H_t \ne \{0\}$ and there is a non-zero element $f=(0,\langle f_i\colon i \in I\rangle)\in \bigcap_{t\in T} N_t^I.$ For every index $i \in I,$ $f_i \in H_t$ ($t \in T$) which entails $f_i=e$ for each $i \in I$ and $f=(0,\langle f_i\colon i \in I\rangle) = 0,$ which is a contradiction. Therefore, $G$ is subdirectly irreducible.
\end{proof}

Now we present a useful notion introduced in \cite{DvKo, DvuK} which will be used in the next result. Let $\phi:\to I$ be a bijection. If, for two indices $i,j \in I$, there is an integer $m\in \mathbb Z$ such that  $\phi^m(i)= j$, then $i$ and $j$ are said to be {\it connected}; otherwise, they are called {\it disconnected}.
If all distinct elements of a subset $C$ of $I$ are connected, $C$ is said to be a {\it connected component} of $I$, and $I$ can be decomposed into a system of maximal connected components. We denote by $\mathcal C(I)$ the system of all connected components. We assert that $\phi^{-1}(C)=C$. Indeed, take $j \in \phi^{-1}(C)$. There is $i \in C$ such that $\phi(j)=i$. Now let $i_1$ be an arbitrary index from $C$. There is an integer $m \in \mathbb Z$ such that $\phi^m(i_1)=i$ which yields $\phi^m(i_1)=\phi(j)$ so that $j = \phi^{m-1}(i)$ proving $j \in C$. Now let $i_0 \in C$, then $j:=\phi(i_0)$ is connected with $i_0$ and $j \in \phi^{-1}(C)$.

\begin{theorem}\label{th:6.6}
Let $I$ be a set and $\phi:I \to I$ be a bijection and let $G$ be a directed non-trivial po-group with \RDP$_1.$ Let $\Gamma(\mathbb Z \lsemiphi G, u_n)$ be the kite $n$-perfect pseudo effect algebra corresponding to the po-group $G.$ The following are equivalent:

\begin{enumerate}
\item[{\rm (1)}] $G$ is subdirectly irreducible and for all $i,j\in I$ there exists an integer $m\ge 0$ such that $\phi^{m}(i)=j$ or $\phi^{-m}(i)=j.$
\item[{\rm (2)}] The kite $n$-perfect pseudo effect algebra is  subdirectly irreducible.
\end{enumerate}
\end{theorem}

\begin{proof}

(1) $\Rightarrow$ (2). By Theorem \ref{th:3.4}--\ref{th:3.5}, the kite $n$-perfect pseudo effect algebra $\Gamma(\mathbb Z \lsemiphi G, u_n)$ satisfies RDP$_1.$

Let $H$ be the least non-trivial o-ideal of $G$ and let $N=H^+.$
By Proposition \ref{pr:6.3}, $N^I$ and $N_f^I$ are normal ideals of the kite $n$-perfect pseudo effect algebra  $\Gamma(\mathbb Z \lsemiphi G, u_n)$. We assert that $N^I_f$ is the least non-trivial normal ideal of the kite. To show that, we prove that the normal ideal of the kite generated by any non-zero element $(0,\langle f_i\colon i \in I\rangle)\in N^I_f$ equals to $N^I_f.$ Or equivalently, we prove the same for any one-dimensional element from $N^I_f.$ Indeed,
let $f=(0,\langle f_i\colon i \in I\rangle)$ be any element from $N^I\setminus \{0\}.$ There is a one-dimensional element $g=(0,\langle g_i\colon i \in I\rangle)\in N^I$ such that $0<\langle g_i\colon i \in I\rangle \le \langle f_i\colon i \in I\rangle.$ Hence, $N^I_f=N_0(g)\subseteq N_0(f)\subseteq N^I_f,$ where $N_0(g)$ and $N_0(f)$ are normal ideals of $\Gamma(\mathbb Z \lsemiphi G, u_n)$ generated by $f$ and $g$, respectively.

Now choose an arbitrary one-dimensional element $g$ from $N^I_f$ and
let $N_0(g)$ be the normal ideal of $\Gamma(\mathbb Z \lsemiphi G, u_n)$ generated by the element $g.$
Without loss of generality assume $g=(0,\langle g_0,e,\ldots\rangle),$ where $g_0>e,$ $g_0 \in G$; this is always possible by a suitable re-indexing of $I,$ regardless of its cardinality. Then $g_0$ generates $N$ while $H$ is the least non-trivial o-ideal of $G.$

For any element $x \in [(0,e^I),(1,e^I)]$, we define two relative complements: $x^{-'}:=(1,e^I)*_{\phi} x^{-1}$ and $x^{\sim'}:=x^{-1}*_\phi (1,e^I)$; both elements are from $E_1$.

Doing double relative negations $m$ times of $(0,\langle f_i\colon i \in I\rangle)$, we obtain that  either $ (0,\langle f_i\colon i \in I\rangle)^{-'-'m}=(0,\langle f_{\phi^{m}(i)}\colon i \in I\rangle)$ and it belongs to $N_0(g)$ or $(0,\langle f_i\colon i \in I\rangle)^{\sim'\sim' m}=(0,\langle f_{\phi^{-m}(i)}\colon i \in I\rangle)$ which also belongs to $N_0(g).$  Consequently, for any $i \in I,$ there is an integer $m$ such that $\phi^{m}(i)=0$ or $\phi^{-m}(i)=0,$ so that the one-dimensional element whose $i$-th coordinate is $g_0$ is defined in $N_0(g)$ for any $i \in I$; it is either $g^{-'-'m}$ or $g^{\sim'\sim' m}.$ We see that $f^{-1}g_0f$ and $kg_0k^{-1}$ belong to $N_0(g_0)$ for all $g,k\in G^+,$ which yields, for every $g \in N,$ the one-dimensional element $(0,\langle g,e,\ldots \rangle)$ belongs to $N_0(g_0),$ and finally, every one-dimensional element $(0,\langle \ldots, g,    \ldots \rangle)$  from $N^I_f$ belongs also to $N_0(g_0).$

Now let $I_0=\{i_1,\ldots,i_m\}$ be an arbitrary finite subset of $I$, and choose arbitrary $m$ elements $h_{i_1},\ldots,h_{i_m} \in N.$ Define an $m$-dimensional element $g_{I_0}=(0,\langle g_i\colon i \in I\rangle),$ where $g_i=h_{i_k}$ if $i=i_k$ for some $k=1,\ldots,m,$ and $g_i=e$ otherwise. In addition, for $k=1,\ldots,m,$ let $\bar g_k=(0,\langle f_i\colon i \in I\rangle),$ where $f_i=h_{i_k}$ if $i=i_k$ and $f_i=e$ if $i\ne i_k.$ Then $g_{I_0}=\bar g_1+\cdots+\bar g_k\in N_0(g_0).$

Consequently, $N_0(g_0)=N_f^I.$

(2) $\Rightarrow$ (1). By Proposition \ref{pr:6.5}, $G$ has the least non-trivial o-ideal, say $H_0$ and let $N_0=H_0^+.$ Suppose that (1) does not hold. Then for all
$i,j\in I$ and every integer $m\ge 0,$ we have $\phi^m(i)\ne j$ and $\phi^{-m}(i)\ne j.$ By the assumption, there are two elements $i_0,j_0\in I$ which are disconnected.  Let $I_0$ and $I_1$ be  maximal sets of mutually connected elements containing $i_0$ and $j_0,$ respectively. Then no element of $I_0$ is connected to any element of $I_1.$

We define $N_0^{I_0}$ as the set of all elements $(0,\langle f_i\colon i \in I\rangle)$ such that $i\notin I_0$ implies $f_i=e.$ In a similar way we define $N^{I_1}.$ Then both sets are non-trivial normal ideals of the kite $n$-perfect pseudo effect algebra.

On the other hand, we have $N_0^{I_0} \cap N_0^{I_1} = \{(0,e^I)\}$ which contradicts that the kite $n$-perfect pseudo effect algebra has the least non-trivial normal ideal.
\end{proof}

If $I$ is a finite set, the subdirectly irreducible kite $n$-perfect pseudo effect algebra has the following form:

\begin{theorem}\label{th:6.7}
Let $I=\{0,1,\ldots,m-1\}$ for some $m\ge 0,$ $\lambda,\rho: I \to I$ be bijections and $G$ be a non-trivial directed po-group with \RDP$_1.$ If the kite $n$-perfect pseudo effect algebra $K^{\phi}_{I,n}(G)$ is subdirectly irreducible, then $G$ is subdirectly irreducible and $K^{\phi}_{I,n}(G)$ is isomorphic to one of:
\begin{enumerate}
\item[{\rm (1)}] $\Gamma(\frac{1}{n}\mathbb Z,1),$ if $m=0,$ $\Gamma(\mathbb Z\lex G,(1,0))$ if $m=1.$

\item[{\rm (2)}] $K^{\phi}_{I,n}(G)$ for $m\ge 2$ and $\phi(i) = i-1\ (\mathrm{mod}\, m).$

\end{enumerate}
\end{theorem}

\begin{proof}
If $I$ is empty, the only bijection from $I$ to $I$ is the empty function. The kite $n$-perfect pseudo effect algebra $K_{\emptyset,n}^\emptyset(G)$ is an $(n+1)$-element linear effect algebra. If $m=1,$ the kite $n$-perfect pseudo effect algebra $K_{I,n}^{Id_I}(G)$ is isomorphic to the symmetric pseudo effect algebra $\Gamma(\mathbb Z\lex G,(1,0)).$

For $m\ge 2,$ we can assume that $I=\{0,1,\ldots,m-1\}$ and $\phi$ is a permutation on $I$. If $\phi$ is not cyclic, then there are $i,j \in I$ such that $j$ does not belongs to the orbit of $i,$ which means that $i$ and $j$ are not connected which contradicts Theorem \ref{th:6.6}. So $\phi$ must be cyclic. We can then renumber $I=\{0,1,\ldots,m-1\}$ following the $\phi$-cycle, so that $\phi(i) = i-1\ (\mathrm{mod}\, m),$ $i =0,1,\ldots,m-1.$
\end{proof}

In the next result we show that, for any subdirectly irreducible $n$-perfect pseudo effect algebra $K^\phi_{I,n}(G)$ with infinite $I$, $I$ has to be countable.

\begin{lemma}\label{le:6.8}
Let $I$ be an infinite set, $\phi:I \to I$ be a bijection, and let $G$ be a non-trivial directed po-group with \RDP$_1.$ If the kite $n$-perfect pseudo effect algebra $K^{\phi}_{I,n}(G)$ is subdirectly irreducible, then $I$ is at most countable.
\end{lemma}

\begin{proof}
Suppose $I$ is uncountable and choose an element $i\in I.$ Consider the set $P(i)=\{\phi^m(i):\ m \ge 0\} \cup \{\phi^{-m}(i): m\ge 0\}.$ The set $P(i)$ is at most countable, so there is $j \in I\setminus P(i).$ But $P(i)$ exhausts all finite paths alternating $\phi$ starting from $i$. Then $i$ and $j$ are disconnected which contradicts Theorem \ref{th:6.6}. Hence, $I$ is at most countable.
\end{proof}

If $|I|=\aleph_0$, subdirectly irreducible kite $n$-perfect pseudo effect algebras are the following form:

\begin{theorem}\label{th:6.9}
Let $|I|=\aleph_0,$ $\phi: I \to I$ be a bijection, and let $G$ be a non-trivial directed po-group with \RDP$_1.$ If the kite $n$-perfect pseudo effect algebra $K^{\phi}_{I,n}(G)$ is subdirectly irreducible, then  $K^{\phi}_{I,n}(G)$ is isomorphic to $K^{\phi}_{\mathbb Z,n}(G)$ where $\phi(i)=i-1,$ $i \in \mathbb Z.$
\end{theorem}

\begin{proof}
If $\phi$ is not cyclic, there would be two elements which should be disconnected which is impossible. Therefore, there is an element $i_0\in I$ such that the orbit $P(i_0):=\{\phi^m(i_0): m \in \mathbb Z\}=I.$ Hence, we can assume that $I=\mathbb Z,$ and $\phi(i)=i-1,$ $i \in \mathbb Z$; indeed, if we set $j_m=\phi^{-m}(i_0),$ $m\in \mathbb Z,$ then $\phi(j_m)=j_{m-1}$ and we have $\phi(i)=i-1,$ $i \in \mathbb Z.$
\end{proof}

\begin{lemma}\label{le:7.4}
Let $K^{\phi}_{I,n}(G)$ be a kite $n$-perfect pseudo effect algebra of a directed po-group $G$ satisfying \RDP$_1.$ Then
$K^{\phi}_{I,n}(G)$ is a subdirect
product of the system of kite $n$-perfect pseudo effect algebras $(K^{\phi'}_{I',n}(G)\colon I'\in \mathcal C(I))$, where $I'$ is any
connected component of $I$,  and $\phi: I'\to I'$ is the restrictions of $\phi$ onto $I'$.
\end{lemma}

\begin{proof}
Let $I'$ be a connected component of $I$. As it was mentioned just before Theorem \ref{th:6.6}, we see that every restriction $\phi'$ of $\phi$ onto  $I'$ of $I$, $\phi':I'\to I'$, is a bijection. Let $I'$ be a connected component of $I$.
Let $N_{I'}$ be the set of all elements
$f = (0,\langle f_i: i \in I\rangle)$ with $f_i\ge 0$ for each $i\in I$
such that  $f_i = e$ whenever $i\in I'$.
It is straightforward to see that $N_{I'}$ is a normal ideal
of $K^{\phi}_{I,n}(G)$. It is also not difficult
to see that $K^{\phi}_{I,n}(G)/N_{I'}$ is isomorphic to
$K^{\phi'}_{I',n}(G)$.

Now, let $\mathcal C(I)$ be the set of all connected components of $I$, and
for each $I'\in \mathcal C(I),$ let $N_{I'}$ be the normal filter defined as above.
As connected components are disjoint, we have
$\bigcap_{I'\in \mathcal C(I)} N_{I'} = \{(0,e^I)\}.$ This proves $K^{\phi}_{I,n}(G)\leq \prod_{I'\in \mathcal C(I)} K^{\phi'}_{I',n}(G).$
\end{proof}

Now we are ready to present an analogue of the Birkhoff representation theorem, \cite[Thm
II.8.6]{BuSa}, for kite $n$-perfect pseudo effect algebras with RDP$_1$. 

\begin{theorem}\label{th:7.5}
Every kite $n$-perfect pseudo effect algebra with \RDP$_1$ is  a subdirect product of subdirectly irreducible kite $n$-perfect pseudo effect algebras with \RDP$_1.$
\end{theorem}

\begin{proof}
Consider a kite $n$-perfect pseudo effect algebra $K^{\phi}_{I,n}(G)$. By Theorems \ref{th:3.4}--\ref{th:3.5}, $G$ satisfies RDP$_1.$ If the kite is not subdirectly
irreducible, then Theorem \ref{th:6.6},  we have two possible situations:
(i) $G$ is not subdirectly irreducible, or
(ii) $G$ is subdirectly irreducible but
there exist $i,j\in I$ such that, for every
$m\in\mathbb N,$ we have $\phi^m(i)\neq j$ and
$\phi^{-m}(i)\neq j$. Observe that this happens if and only if
$i$ and $j$ do not belong to the same connected component of $I$.

By \cite[Lem 3.4]{DvHo}, every directed po-group with RDP$_1$ is a subdirect product of subdirectly irreducible po-groups with RDP$_1$. Hence,
we can reduce (i) to (ii). So, suppose $G$ is subdirectly irreducible.
Then, using Lemma~\ref{le:7.4}, we can subdirectly embed
$K^{\phi}_{I,n}(G)$ into
$\prod_{I'} K^{\phi'}_{I',n}(G)$, where $I'$
ranges over the connected components of $I,$  and $\phi'$ is the restriction of $\phi$ onto $I'$. But then, each
$K^{\phi}_{I,n}(G)$ is subdirectly irreducible in view of Theorem \ref{th:6.6}, and by Theorems \ref{th:3.4}--\ref{th:3.5}, it satisfies RDP$_1.$
\end{proof}

\section{Conclusion}

In the paper we have extended the study of a special class of pseudo effect algebras called kite pseudo effect algebras, or simply kites, originally introduced in \cite{DvuK} and motivated by the research in \cite{DvKo}. They are connected with a fixed po-group $G$, an index set $I$ and with two bijections $\lambda,\rho:I \to I$, and the elements of the kite are from $G^I$. Using different bijections, the resulting algebra can be non-commutative. The construction resembles the so-called wreath product, \cite{Dar, Gla}.  In \cite{DvuK}, there was formulated an open problem: Characterize a unital po-group $(H,u)$ with a special kind of RDP such that the given kite is isomorphic to some interval in the positive cone of $(H,u)$. The solution is given by a special kind of the lexicographic extension of $G$, and the main effort was done showing that this extension satisfies RDP$_1$, see Theorems \ref{th:3.4}--\ref{th:3.5}. In a special case when $G$ satisfies RDP$_2$, the solution is given, too, Theorem \ref{th:3.7}.

We extend the notion of kite pseudo effect algebras introducing the class of $n$-perfect pseudo effect algebras. We characterize subdirectly irreducible algebras which are basic building bricks of the theory of kite perfect pseudo effect algebras, see Theorems \ref{th:6.6}, \ref{th:6.9}, and \ref{th:7.5}.

Finally, we have formulated a problem to characterize a pseudo effect algebra to be  isomorphic to some kite $n$-perfect pseudo effect algebra.

\end{document}